\documentclass[a4paper]{article}

\RequirePackage{amsthm,amsmath}
\RequirePackage[numbers]{natbib}

\usepackage{amssymb}
\usepackage{latexsym}
\usepackage{color}
\usepackage{xr}

\usepackage{amsfonts}
\usepackage{epsfig}
\usepackage[all]{xy}
\usepackage{enumitem}
\usepackage{multirow}
\usepackage{float}
	
\usepackage{comment}
\excludecomment{discuss}

\newcommand{\colord}{\color{black}}
\newcommand{\colorp}{\color{black}}
\newcommand{\colora}{\color{black}}

\newtheorem{lemma}{Lemma}[section]
\newtheorem{proposition}{Proposition}[section]
\newtheorem{theorem}{Theorem}[section]
\newtheorem{remark}{Remark}[section]

\def\td{\mathrm{d}}

\def\F{\mathcal{F}}
\def\R{\mathbb{R}}

\newcommand{\EQU}[1]{\begin{equation}{#1}\end{equation}}
\newcommand{\EQQ}[1]{\begin{equation}{#1 \nonumber}\end{equation}}
\newcommand{\EQN}[1]{\begin{equation}\begin{split}{#1}\end{split}\end{equation}}
\newcommand{\EQNN}[1]{\begin{equation}\begin{split}{#1 \nonumber}\end{split}\end{equation}}
 
\newcommand{\DIV}[2]{\left\{\begin{array}{#1} #2 \end{array} \right.} 

\newcommand{\mbbn}{\mathbb{N}}
\newcommand{\PX}{\partial_x}
\newcommand{\PY}{\partial_y}
\newcommand{\PZ}{\partial_z}
\newcommand{\PT}{\partial_\theta}
\newcommand{\mfa}{\mathfrak{A}}
\newcommand{\mcy}{\mathcal{Y}}
\newcommand{\SUML}{\sum_{l=1}^{L_n-1}}
\newcommand{\CHEH}{\check{H}_{n,\delta}}
\newcommand{\TILH}{\tilde{H}_{n,\delta}}

\begin{document}

\title{Efficient drift parameter estimation for ergodic solutions of backward SDEs}
\author{Teppei Ogihara$^*$ and Mitja Stadje$^{**}$\\
$*$
\begin{small}
Graduate School of Information Science and Technology, University of Tokyo, Tokyo, Japan, ogihara@mist.i.u-tokyo.ac.jp
\end{small}\\
$**$
\begin{small}
Faculty of Mathematics and Economics, University of Ulm, Ulm, Germany, mitja.stadje@uni-ulm.de
\end{small}
}
\maketitle


\noindent
{\bf Abstract.}
We derive consistency and asymptotic normality results for quasi-maximum likelihood methods for drift parameters of ergodic stochastic processes observed in discrete time in an underlying continuous-time setting. The special feature of our analysis is that the stochastic integral part is unobserved and non-parametric. Additionally, the drift may depend on the (unknown and unobserved) stochastic integrand. Our results hold for ergodic semi-parametric diffusions and backward SDEs. Simulation studies confirm that the methods proposed yield good convergence results.\\

\section{Introduction}

	The paper analyzes statistical inference for Markovian ergodic forward backward stochastic differential equations (BSDEs). Ergodic {\colora solutions of} backward SDEs may be seen as a generalization of an ergodic Markovian diffusion process with unknown but ergodic diffusion part. Specifically, consider a probability space $(\Omega,\mathcal{F}=(\mathcal{F}_t)_t,P)$ with filtration $\mathcal{F}$ being generated by a $d$-dimensional Brownian motion $W$. Let $Y$ be a $d$-dimensional Markov diffusion process depending on an unknown parameter $\theta\in \mathbb{R}^m$. $Y$ will in the sequel be also referred to as a data generating process. In the classical statistical inference problem for stochastic processes $Y$ satisfies a stochastic differential equation of the form
		\begin{align} 
		\label{basic}
dY_t=	\psi(t,Y_t,\theta) + \sigma(t,Y_t,\theta) dW_t
	\end{align}
	where $\psi$ and $\sigma$ are known functions and $Y_0$ is assumed to be known as well. A classical example for $Y$ is given by a Brownian motion with drift or, rather popular in finance, a geometric Brownian motion. Statistical inference results for (\ref{basic}) are analyzed through quasi-maximum likelihood methods in {\colord Yoshida (1992, 2011), Kessler (1997) and Uchida and Yoshida (2012). They have been extended to jump--diffusion processes by Shimizu and Yoshida (2006) and Ogihara and Yoshida (2011).} Now assume that the diffusion function $\sigma$ in \eqref{basic} is \emph{unknown} and that we only know that the integrand of the diffusion part is given by a positive definite $\mathbb{R}^{d\times d}$-valued ergodic predictable process, say {\colorp $V_tV_t^\intercal$} bounded away from zero. This leads to the stochastic differential equation 
	\begin{align*} 
dY_t=	\psi(t,Y_t,\theta) + {\colorp V_t dW_t}
\end{align*}
	where   $V$ may be identified with a triagonal ergodic stochastic process. Next, suppose that we additionally allow the integrand of the drift, $\psi,$ to possibly also depend on $V_t V^\intercal_t $ and furthermore on an observed additional Markov process $X$. Then we have that $Y$ satisfies
	\begin{align} 
	\label{ebsde}
dY_t=	\psi(t,X_t,Y_t,V_t V^\intercal_t ,\theta)dt+  V_t dW_t.
	\end{align} 
This equation is also called {\colora a} backward stochastic differential equation with solution $(Y,V)$ and driver function $\psi$. The goal of this paper to give consistency and asymptotic normality results to estimate $\theta$ in (\ref{ebsde}) with data generating processes $(Y,X)$ and discrete time observations. 

BSDEs have been introduced by Peng and Pardoux (1991) and have since been extended in many directions regarding assumptions on the driver function, connections to PDEs and Hamilton-Jacobi-Bellman equations, applications to stochastic optimal control theory, smoothness of $(Y,V)$, robustness, numerical approximations and invariance principles. Although originally developed for a finite maturity, in many situations the terminal time is either random or there is no natural terminal time at all and the decision maker faces instead an infinite time horizon. Usually in the theory of BSDEs existence and uniqueness of a solution can be guaranteed by Lipschitz conditions on the driver. Now for an infinite time horizon {\colora the BSDE} may be ill posed which has been addressed by Briand and Hu (1998) by imposing a monotonicity assumption on the driver. However, for our statistical analysis we will simply assume that the data generating process satisfies an equation of the form (\ref{ebsde}) and is ergodic.
{\colora In this case we refer to (\ref{ebsde}) also as an ergodic BSDE.} 

Ergodic backward SDEs for finite or infinite dimensional Brownian motion have for instance been considered in Buckdahn and Peng (1999), Fuhrmann, Hu and Tessitore (2009), Richou (2009), Debussche, Hu and Tessitore (2011), Hu and Wang (2018), Madec (2015), Hu et al. (2015), Liang and Zariphopoulou (2017), Chong et al. (2019), Hu and Lemonnier (2019), Hu, Liang and {\colora Tang} (2020) and Guatteri and Tessitore (2020).

For statistical inference on BSDEs there is in general not much literature available. For nonparametric estimation of linear drivers see Su and Lin (2009), Chen and Lin (2010) and Zhang (2013). Zhang and Lin (2014) propose two terminal dependent estimation methods for integral forms of backward SDEs. Song (2014) gives results under independence assumptions. These works consider BSDEs which are non-ergodic and therefore need additional assumptions. In this work we show asymptotic results instead for an infinite time horizon under ergodicity assumptions on $(Y,V,X)$.   
{\colord Even if limited to conventional SDEs, our results enables drift parameter estimaion with an unknown volatility process, unlike previous studies (see Example 1 in Section~\ref{se:examples}).}

The paper is structured as follows: In Section \ref{se:2} we describe the setting our assumptions and give the main results. Section \ref{se:examples} gives a number of applications and examples. Section \ref{se:simulation} contains numerical studies in the one- and multidimensional case. The proofs can be found in Section \ref{se:proofs}.

\section{Main results}\label{se:2}

Given a probability space $(\Omega,\F,P)$ with a right-continuous filtration
$\mathbf{F}=(\F_t)_{t\geq 0}$,
let $Y=(Y_t)_{t\geq 0}$ be a $d_Y$-dimensional $\mathbf{F}$-adapted process satisfying
\begin{equation*}
Y_t=Y_T-\int_t^T\psi(X_s,Y_s,V_sV_s^\intercal,\theta_0)ds-\int_t^TV_sdW_s, \quad 0\leq t\leq T<\infty,
\end{equation*}
where $W=(W_t)_{t\geq 0}$ is a $d_W$-dimensional standard $\mathbf{F}$-Wiener process ($d_W\geq d_Y$),
$\theta_0\in\Theta$ is an unknown parameter, $\Theta$ is a bounded open subset in $\R^d$, $\psi$ is an $\mathbb{R}^{d_Y}$-valued function, $X=(X_t)_{t\geq 0}$ is a $d_X$-dimensional {\colord continuous} $\mathbf{F}$-adapted process, 
$V=(V_t)_{t\geq 0}$ is a $d_Y\times d_W$ matrix-valued {\colord continuous} $\mathbf{F}$-adapted process.
{\colord The dimension $d_X$ of $X_t$ is possibly zero. In that case, we ignore $X_t$.}
We observe $\{(X_{kh_n},Y_{kh_n})\}_{k=0}^n$, and consider asymptotics: $h_n\to 0, nh_n\to \infty$ and $nh_n^2\to 0$ as $n\to\infty$.

We construct a maximum-likelihood-type estimator for the parameter $\theta_0$.
For this purpose, we construct a quasi-likelihood function $H_n(\theta)$.
{\colord Let $\Delta_l U=U_{t^{l+1}_0}-U_{t^l_0}$ for a stochastic process $(U_t)_{t\geq 0}$.}
Let $(c_n)_{n\in \mbbn}$ be a sequence of positive integers such that 

\EQU{\label{cn-cond}{\colorp  c_nn^{-\epsilon}\to \infty \quad {\rm and} \quad c_nh_nn^{\epsilon}\to 0,}}
for some $\epsilon>0$.
Let $L_n=[n/c_n]$, $t^l_m=(m+c_nl)h_n$, and let
\begin{equation*}
\hat{Z}_l=\frac{1}{c_nh_n}\sum_{m=1}^{c_n}(Y_{t^l_m}-Y_{t^l_{m-1}})(Y_{t^l_m}-Y_{t^l_{m-1}})^\intercal \quad (0\leq l\leq L_n-1),
\end{equation*}
where $\intercal$ denotes transpose.
We define a quasi-log-likelihood function by 
  \EQU{{\colord H_n(\theta)=-\frac{1}{2}\sum_{l=1}^{L_n-1}\bigg\{(\Delta_l Y-c_nh_n\hat{\psi}_l(\theta))^\intercal\frac{\hat{Z}_{l-1}^{-1}}{c_nh_n}
   (\Delta_l Y-c_nh_n\hat{\psi}_l(\theta))\bigg\}1_{\{\det \hat{Z}_{l-1}>0\}},}
  }
where $\bar{\Theta}$ is the closure of $\Theta$ and $\hat{\psi}_l(\theta)=\psi(X_{t^l_0},Y_{t^l_0},\hat{Z}_{l-1},\theta)$.
Let $\Delta_l U=U_{t^{l+1}_0}-U_{t^l_0}$ for a stochastic process $(U_t)_{t\geq 0}$.

Then we can construct a maximum-likelihood-type estimator $\hat{\theta}_n$ as a random variable which maximizes $H_n$; 
$\hat{\theta}_n\in{\rm argmax}_{\theta\in\bar{\Theta}} H_n(\theta)$.\\

Let $\mathfrak{P}$ be the space of $d_Y\times d_Y$ {\colorp symmetric,} {\colord positive} definite matrices. 
{\colord For a vector $v=(v_i)_{1\leq i\leq k}$ and a matrix $m=(m_{ij})_{\substack{1\leq i\leq k_1 \\ 1\leq j\leq k_2}}$, we denote
 \EQQ{\partial_v^l=(\frac{\partial^l}{\partial v_{i_1}\cdots \partial v_{i_l}})_{i_1,\cdots, i_l=1}^k 
   \quad {\rm and} \quad  
   \partial_m^l=(\frac{\partial^l}{\partial m_{i_1j_1}\cdots \partial m_{i_lj_l}})_{\substack{1\leq i_1,\cdots, i_l\leq k_1 \\ 1\leq j_1,\cdots, j_l\leq k_2}}.}}
We assume that $\Theta$ admits Sobolev's inequality, that is, for any $p>d$,
there exists a positive constant $C_p$ depending only $p$ and $\Theta$ such that 
\EQU{\label{sobolev} \sup_{x\in\Theta}|u(x)|\leq C\sum_{k=0,1}\bigg(\int_\Theta |\PX^ku(x)|^pdx\bigg)^{1/p}}
for any $u\in C^1(\Theta)$.
Sobelev's inequality is satisfied if $\Theta$ has a Lipschitz boundary (see Adams and Founier~(2003)).

{\colord Let $\bar{\mathfrak{P}}$ be the closure of $\mathfrak{P}$ in $\R^{d_Y}\otimes \R^{d_Y}$, and $\mathfrak{P}_\delta=\{z\in \mathfrak{P}|z-\delta I \in \mathfrak{P}\}$ for any $\delta>0$, where $I$ is the unit matrix.}
{\colorp For $p\geq 1$ and $r\geq 1$, we consider the following assumptions.}
\begin{description}
\item[{\colorp Assumption (A1-$p$).}] $\sup_{t\geq 0}\lVert (V_tV_t^\intercal)^{-1}\rVert<\infty$ almost surely and {\colorp there exists a positive constant $C$ such that
\begin{eqnarray}
E[|V_t-V_s|^{2p}]^{1/(2p)}+E[|X_t-X_s|^p]^{1/p}&\leq & C|t-s|^{1/2}, \nonumber \\
E\bigg[\bigg|\frac{E[V_t-V_s|\mathcal{F}_s]}{(t-s)}\bigg|^{2p}\bigg]&\leq &C, \nonumber \\
E[|X_s|^p]\vee E[|V_s|^{2p}]\vee E[|Y_s|^p]&\leq&C, \nonumber 
\end{eqnarray}
for $0\leq s<t$.}
\item[{\colorp Assumption (A2-$r$).}]
$\PT^l\psi(x,y,z,\theta)$ exists and is continuous on $\mathbb{R}^{d_X}\times \mathbb{R}^{d_Y}\times{\colord \bar{\mathfrak{P}}}\times \bar{\Theta},$
for $l\in \{0,1,2\}$, 
and there exists a {\colorp constant $C$} such that
\begin{equation*}
|\PT^l\psi(x,y,z,\theta)|\leq {\colorp C(1+|x|+|y|+|z|)^r}.
\end{equation*}
{\colorp Moreover, for any $\delta>0$, there exists a constant $C_\delta$ such that} 
\EQNN{
&|\PT^l\psi(x_1,y_1,z_1,\theta)-\PT^l\psi(x_2,y_2,z_2,\theta)| \\
&\quad \leq {\colorp C_\delta(1+|x_1|+|y_1|+|z_1|)^r}(|x_1-x_2|+|z_1-z_2|+|y_1-y_2|)}
for $l\in\{0,1,2\}$, $x,x_1,x_2\in \R^{d_X}$, $y,y_1,y_2\in\R^{d_Y}$, {\colord $z\in \bar{\mathfrak{P}}$, $z_1,z_2\in \mathfrak{P}_\delta$}, and $\theta\in\Theta$.
\item[{\colorp Assumption (A3-$p$).}]
At least one of the following two conditions holds true.
\begin{enumerate}
\item The function $\psi(x,y,z,\theta)$ does not depend on $y$ and $(X_t,V_tV_t^\intercal)$ is ergodic, that is, there exists an invariant distribution $\pi$ such that
for any measurable function $f$,
\begin{equation*}
\frac{1}{T}\int^T_0f(X_t,V_tV_t^\intercal)dt\overset{P}\to \int f(x,z)\pi(dxdz),
\end{equation*}
as $T\to\infty$. {\colorp Moreover, 
\begin{equation*}
\int \bigg(\frac{1+|x|+|z|}{(\det z) \wedge 1}\bigg)^p\pi(dxdz)<\infty.
\end{equation*}
}
\item $(X_t,Y_t,V_tV_t^\intercal)$ is ergodic, that is, there exists an invariant distribution $\pi$ such that
for any measurable function $f$,
\begin{equation*}
\frac{1}{T}\int^T_0f(X_t,Y_t,V_tV_t^\intercal)dt\overset{P}\to \int f(x,y,z)\pi(dxdydz),
\end{equation*}
as $T\to\infty$. {\colorp Moreover,  
\begin{equation*}
\int \bigg(\frac{1+|x|+|y|+|z|}{(\det z) \wedge 1}\bigg)^p\pi(dxdydz)<\infty.
\end{equation*}
}
\end{enumerate}
\item[Assumption (A4).] (Identifiability condition)
For $\theta_1,\theta_2\in \bar{\Theta}$, $\psi(x,y,z,\theta_1)=\psi(x,y,z,\theta_2)$ for all $(x,y,z)$ on ${\rm supp}(\pi)$ implies $\theta_1=\theta_2$.
\end{description}

{\colord Most of the above assumptions are {\colora standard} for asymptotic theory of maximum-likelihood-type estimation to ergodic diffusion processes, and similar (or stronger) assumptions are required in Kessler (1997) and Uchida and Yoshida (2012).
A similar statement applies to Condition (A2$'$-$r$) appearing later.
Here, the upper bound $C_\delta$ of $\PT^l \psi$ in (A2-$r$) depends on $\delta$. While this assumption is not {\colora a} typical one, 
by doing so, (A2-$r$) is satisfied even the case that $\psi$ is not smooth at $z=0$ (for example, $\psi(x,y,z,\theta)=\theta \sqrt{z}1_{\{z>0\}}$ with $d_Y=1$).
For sufficient conditions of ergodicity for $(X_t,Y_t,V_tV_t^\intercal)$,
we refer readers to Remark 1 of Uchida and Yoshida (2012).

{\colorp Fix $\epsilon>0$ satisfying (\ref{cn-cond}).}
Under the assumptions above, we obtain consistency of our estimator.
}

\begin{theorem}[consistency]\label{consistency-thm}
{\colorp Let $p,r\geq 1$ such that 
\EQU{\label{pr-cond} \frac{p}{4r}>d\vee \frac{2}{\epsilon} \vee 4.}
Assume (A1-$p$), (A2-$r$), (A3-$p$), and (A4).}
Then $\hat{\theta}_n\overset{P}\to \theta_0$ as $n\to\infty$.
\end{theorem}

~

Under {\colord (A2-$r$) and} (A3-$p$), we define
\begin{equation*}
\Gamma=\int \partial_\theta \psi(x,z,\theta_0)^\intercal z^{-1}\partial_\theta \psi(x,z,\theta_0)\pi(dxdz)
\end{equation*}
if the function $\psi(x,y,z,\theta)$ does not depend on $y$, and otherwise we define
\begin{equation*}
\Gamma=\int \partial_\theta \psi(x,y,z,\theta_0)^\intercal z^{-1}\partial_\theta \psi(x,y,z,\theta_0)\pi(dxdydz).
\end{equation*}

To deduce asymptotic normality of our estimator, we need a further condition.
{\colord Let $\mathcal{O}$ be an open set in $\R^{d_Y}\otimes \R^{d_Y}$ such that $\bar{\mathfrak{P}}\subset \mathcal{O}$.}
\begin{description}
\item[Assumption (A2$'$-$r$).]
{\colord (A2-$r$) is satisfied.}
$\PX^i\PY^j\PZ^k\PT^l\psi(x,y,z,\theta)$ exists and is continuous on $\mathbb{R}^{d_X}\times \mathbb{R}^{d_Y}\times{\colord \mathcal{O}}\times \bar{\Theta}$ for $l\in\{0,1,2,3\}$ and $i,j,k\in\{0,1,2\}$ with $i+j+k\leq 2$, 
and {\colord for any $\delta>0$,} there exists {\colorp a constant $C'_\delta$} such that
$$|\PX^i\PY^j\PZ^k\PT^l\psi(x,y,z,\theta)|\leq {\colorp C'_\delta(1+|x|+|y|+|z|)^r}$$
for {\colord $x\in\R^{d_X}$, $y\in\R^{d_Y}$, $z\in\mathfrak{P}_\delta$,} $l\in\{0,1,2,3\}$ and $i,j,k\in\{0,1,2\}$ with $i+j+k\leq 2$.

Moreover, there exist a Wiener process $(W'_t)_{t\geq 0}$ independent of $(W_t)_{t\geq 0}$ and $\mathbf{F}$-progressively measurable processes $(a_t^j)_{t\geq 0}$ for $j\in \{1,2,3\}$ such that
$$X_t=X_0+\int^t_0a_s^1ds+\int_0^ta_s^2dW_s+\int_0^ta_s^3dW'_s,$$
and $\sup_{t\geq 0}E[|a_t^j|^p]<\infty$ for any $p>0$ and $j\in\{1,2,3\}$.
\end{description}
\begin{discuss}
{\colorr The estimate for $\PX^i\PY^j\PZ^k\PT^l\psi$ is used in (\ref{psi-diff-eq}) and the estimates for $\PT^l\tilde{\psi}-\PT^l\acute{\psi}$ and $\PT \Lambda_1$ in Proposition~\ref{score-diff-prop}}
\end{discuss}

Suppose that $n^3h_n^5\to 0$. Then we can choose $c_n$ in the definition of $H_n$ satisfying 
\begin{equation}\label{cn-condition}
nh_n^2c_n\to 0 \quad {\rm and} \quad \sqrt{nh_n}/c_n\to 0.
\end{equation}
{\colorp For such $c_n$, fix $\epsilon>0$ satisfying (\ref{cn-cond}).}

\begin{theorem}[Asymptotic normality]\label{asymp-normal-thm}
{\colorp Let $p,r\geq 1$ such that (\ref{pr-cond}) is satisfied.}
Assume (A1-$p$), (A2$'$-$r$), (A3-$p$), (A4), and that $n^3h_n^5\to 0$ as $n\to\infty$. Assume further that $\Gamma$ is positive definite and $c_n$ satisfies (\ref{cn-condition}). Then
\begin{equation*}
\sqrt{nh_n}(\hat{\theta}_n-\theta_0)\overset{d}\to N(0,\Gamma^{-1}).
\end{equation*}
\end{theorem}

{\colord The condition $n^3h_n^5\to 0$ is stronger than the ones in previous works (for instance $nh_n^2\to 0$ in Yoshida (2011), and $nh_n^p\to 0$ for $p\geq 2$ in Uchida and Yoshida (2012) and Kessler (1997)).
Unlike previous studies, we need to construct an estimator $\hat{Z}_l$ of $Z_t$ whose structure is not specified.
For this purpose, (\ref{cn-condition}) and consequently $n^3h_n^5\to 0$ is required.

\begin{remark}
If $V_t$ is a diffusion process with {\colora SDE-coefficients not depending} on $\theta$, $\hat{\theta}_n$ is asymptotically efficient under the assumptions of Gobet (2002) because $\Gamma^{-1}$ corresponds the efficient asymptotic variance in Gobet (2002).
\end{remark}
}

\section{Examples}\label{se:examples}

\begin{enumerate}
\item The first example to which our results apply is a data generating process of the form $$X_0 = x_0,$$ 
$$\td X_t = \psi(t,X_t,\theta)\td t + V_t \td W_t,$$
where $V$ is an unknown predictable ergodic process.
We remark that previous literature only treated the case $dX_t = \psi(t,X_t,\theta)\td t + \sigma(t,X_t,\theta) \td W_t$ with $\psi$ and $\sigma$ known.
\item As a further example consider 
\begin{align}
dP_s &:=
\begin{pmatrix}
dP_s^1 \\ dP_s^2
\end{pmatrix}\nonumber \\
&=
\begin{pmatrix}
 (\mu P^1_s + \sqrt{Z_s^{1,1}} \sqrt{\nu_s} \theta^1)\td s +\sqrt{Z_s^{1,1}} \td W_s^{1} \\
 ( \mu P^2_s + \sqrt{Z_s^{2,1}} \sqrt{\nu_s} \theta^1 + \sqrt{Z_s^{2,2}} \sqrt{\nu_s} \theta^2) \td s +  \sqrt{Z_s^{2,1}} \td W_s^1 + \sqrt{Z_s^{2,2}} \td W_s^2  \label{eq.Heston}
\end{pmatrix} .
\end{align}
with $\mu\leq 0.$
This backward SDE is motivated by extending the evolution of a price process in the Heston model to a random and possibly arbitrary large time horizon.
\item Ergodic BSDEs appear naturally in forward performance processes which are utility functionals which do not depend on the specific time horizon, see for instance Hu, Liang and Tang (2020). In Liang and Zariphopoulou (2017) for instance a forward performance process is desribed which has the factor form $U(x,t)=\frac{x^\delta}{\delta}e^{Y_t-\lambda t}$ with $Y$ being the {\colora ergodic} solution of an BSDE with quadratic driver function.
\end{enumerate}

\section{Simulation studies}\label{se:simulation}

In the sequel, we will consider different possibilities for our sequences converging to zero or to infinity. In particular, consider
$c_n = n^{0.05k}, k = 1,2,\ldots,l-1.$

$h_n = n^{-0.05l}, l= 11,\ldots,19.$
Then we must have
\begin{enumerate}[label=\alph*)]
\item $n h_n^2 c_n = n^{1+0.05k-0.1l} \rightarrow 0$ \\
$\Rightarrow 0.05k < 0.1l-1 \Rightarrow k < 2l-20$
\item $\frac{\sqrt{n h_n}}{c_n} = n^{\frac{1}{2}-0.025l-0.05k} \rightarrow 0$ \\
$\Rightarrow \frac{1}{2}-0.025l < 0.05k \Rightarrow 10-\frac{l}{2} < k$
\item $n^3 h_n^5 \rightarrow 0$ \\
$\Rightarrow n^{3-0.25l} \rightarrow 0$
$\Rightarrow 3-0.25l < 0 \Rightarrow 12<l $
\end{enumerate} 

Combining three cases yields
$13 \leq l \leq 19, \max(1,10-\frac{l}{2}) \leq k \leq \min(19,2l-20) = 2l-20$. We will below try every one of these combinations.\\

\subsection{Simulation Results for the Vasicek model}

Suppose that $X_t$ evolves according to the Vasicek model, that is, $\td X_t=a(b-X_t)\td t + \sigma \td W_t$ where $W_t$ is the standard Brownian motion, with parameters $a=2, b=0.3$ and $\sigma=0.025$. The initial value $X_0$ is set as $0.3$. Let us estimate $\theta$ in the equation
\begin{equation}
\td Y_t = \theta \sqrt{|X_t|+0.1}\ \td t + \sqrt{|X_t|+0.1}\ \td W_t \ ,
\end{equation}
where $Y_0=1$. 

In the following $h_n$ is set to be $n^{-0.05l}$ and $c_n$ to be $n^{0.05k}$. We consider integers $l$ and $k$ where to satisfy the conditions of Theorem \ref{consistency-thm} and Theorem \ref{asymp-normal-thm} $13 \leq l \leq 19$ and $\max(1,10-\frac{l}{2}) \leq k \leq 2l-20$. To look for the pair of $(l,k)$ which best estimates $\theta$, we run simulations for each combination of $(l,k)$ and calculate the average of the errors as the sum of differences between $\hat{\theta}_n$ and $\theta$ in percentage for the $n$'s simulated, which means
$$ \text{Error} = \frac{\sum_{n \in \mathcal{A}} |\hat{\theta}_n - \theta| / \theta}{|\mathcal{A}|},$$
where $\mathcal{A}$ denotes the set of $n$'s simulated.
Two sets of $n$'s are considered: $\mathcal{A}_1 = \{1\times 10^5,2\times 10^5,\ldots,1\times 10^6 \}$ and $\mathcal{A}_2 = \{1\times 10^6,2\times 10^6,\ldots,1\times 10^7 \}$. We let $\theta = 1$.

The results are summarized in the following tables.

\begin{table}[!h]
\centering
\begin{tabular}{|c|c|ccccccc|}
\hline
\multicolumn{2}{|c|}{} & \multicolumn{7}{c|}{$l$} \\
\cline{3-9}
\multicolumn{2}{|c|}{} & 13 & 14 & 15 & 16 & 17 & 18 & 19 \\
\hline
\multirow{18}{*}{$k$}
& 1 &   &   &   &   &   & 4.79986  & 4.39994  \\
& 2 &   &   &   & 0.55168  & 0.59204  & 0.64261  & 0.63193  \\
& 3 &   & 0.13564  & 0.19179  & 0.17408  & 0.43068  & 0.45217  & 0.82545  \\
& 4 & 0.065  & 0.16896  & 0.0839  & 0.21815  & 0.36891  & 0.46106  & 0.86921  \\
& 5 & 0.11211  & 0.14296  & 0.24044  & 0.29471  & 0.30672  & 0.36704  & 0.72157  \\
& 6 & 0.07487  & 0.10097  & 0.21671  & 0.19126  & 0.2234  & 0.44338  & 0.57126  \\
& 7 &   & 0.10343  & 0.16694  & 0.20898  & 0.19727  & 0.48259  & 0.55946  \\
& 8 &   & 0.1056  & 0.22114  & 0.24371  & 0.25512  & 0.63417  & 0.7991  \\
& 9 &   &   & 0.11754  & 0.19612  & 0.29589  & 0.32613  & 0.51654  \\
& 10 &   &   & 0.14666  & 0.17857  & 0.24282  & 0.18316  & 0.56393  \\
& 11 &   &   &   & 0.31039  & 0.22011  & 0.63986  & 0.71099  \\
& 12 &   &   &   & 0.23643  & 0.22018  & 0.31369  & 0.51456  \\
& 13 &   &   &   &   & 0.40641  & 0.50407  & 0.43586  \\
& 14 &   &   &   &   & 0.27931  & 0.50327  & 0.29167  \\
& 15 &   &   &   &   &   & 0.43433  & 0.38009  \\
& 16 &   &   &   &   &   & 0.52718  & 0.41497  \\
& 17 &   &   &   &   &   &   & 0.65534  \\
& 18 &   &   &   &   &   &   & 0.52093  \\

\hline
\end{tabular}
\caption{Errors from different combinations of $(l,k)$ simulated for $ \mathcal{A}_1$.}
\end{table}

\newpage

\begin{table}[!h]
\centering
\begin{tabular}{|c|c|ccccccc|}
\hline
\multicolumn{2}{|c|}{} & \multicolumn{7}{c|}{$l$} \\
\cline{3-9}
\multicolumn{2}{|c|}{} & 13 & 14 & 15 & 16 & 17 & 18 & 19 \\
\hline
\multirow{18}{*}{$k$}
& 1 &   &   &   &   &   & 1.97497  & 2.87813  \\
& 2 &   &   &   & 0.25113  & 0.41856  & 0.53016  & 1.08368  \\
& 3 &   & 0.06392  & 0.16965  & 0.20284  & 0.36099  & 0.47998  & 0.82778  \\
& 4 & 0.05567  & 0.06933  & 0.09849  & 0.1913  & 0.21299  & 0.51363  & 0.61066  \\
& 5 & 0.08798  & 0.06048  & 0.10773  & 0.19639  & 0.27578  & 0.2966  & 0.65836  \\
& 6 & 0.06242  & 0.10747  & 0.10952  & 0.19988  & 0.27281  & 0.32791  & 0.53102  \\
& 7 &   & 0.08838  & 0.12689  & 0.08608  & 0.18994  & 0.23873  & 0.44328  \\
& 8 &   & 0.05909  & 0.16884  & 0.20834  & 0.29658  & 0.44631  & 0.73857  \\
& 9 &   &   & 0.17656  & 0.15423  & 0.28707  & 0.33089  & 0.70613  \\
& 10 &   &   & 0.13615  & 0.21278  & 0.19562  & 0.38462  & 0.58632  \\
& 11 &   &   &   & 0.09943  & 0.15424  & 0.48022  & 0.71004  \\
& 12 &   &   &   & 0.17643  & 0.38302  & 0.32119  & 0.57695  \\
& 13 &   &   &   &   & 0.23213  & 0.22146  & 0.54199  \\
& 14 &   &   &   &   & 0.19692  & 0.47462  & 0.54148  \\
& 15 &   &   &   &   &   & 0.40643  & 0.38577  \\
& 16 &   &   &   &   &   & 0.33587  & 0.66682  \\
& 17 &   &   &   &   &   &   & 0.51917  \\
& 18 &   &   &   &   &   &   & 0.92505  \\

\hline
\end{tabular}
\caption{Errors from different combinations of $(l,k)$ simulated for $n \in \mathcal{A}_2$.}
\end{table}
From the tables it can be seen that the choices for $l$ and $k$ strongly matter. The pairs with $l=13$ gives the smallest error and estimates $\theta$ most accurately under both sets of $n$'s.
When simulations are repeated, any of the three pairs could result in the smallest error. Overall, for the same $k$, the smaller $l$ is, the better the estimation for $\theta$ is. 

Below, Figure \ref{Vasicek} shows an analysis for the Vasicek model where $k$ and $l$ are chosen to be 6 and 13 respectively, with $\theta_0=10$. The number of simulation times $n$ is set as 
\EQQ{\{1000,2000,\ldots,10000,20000,\ldots,100000,200000 , \ldots , 500000\}.} 
For each $n$, we repeat the process by 500 times and calculate the Mean Error of the estimators $\hat{\theta}$'s. 

\subsection{The Heston model}

\begin{figure}[h]
\centering
\includegraphics[scale=0.07]{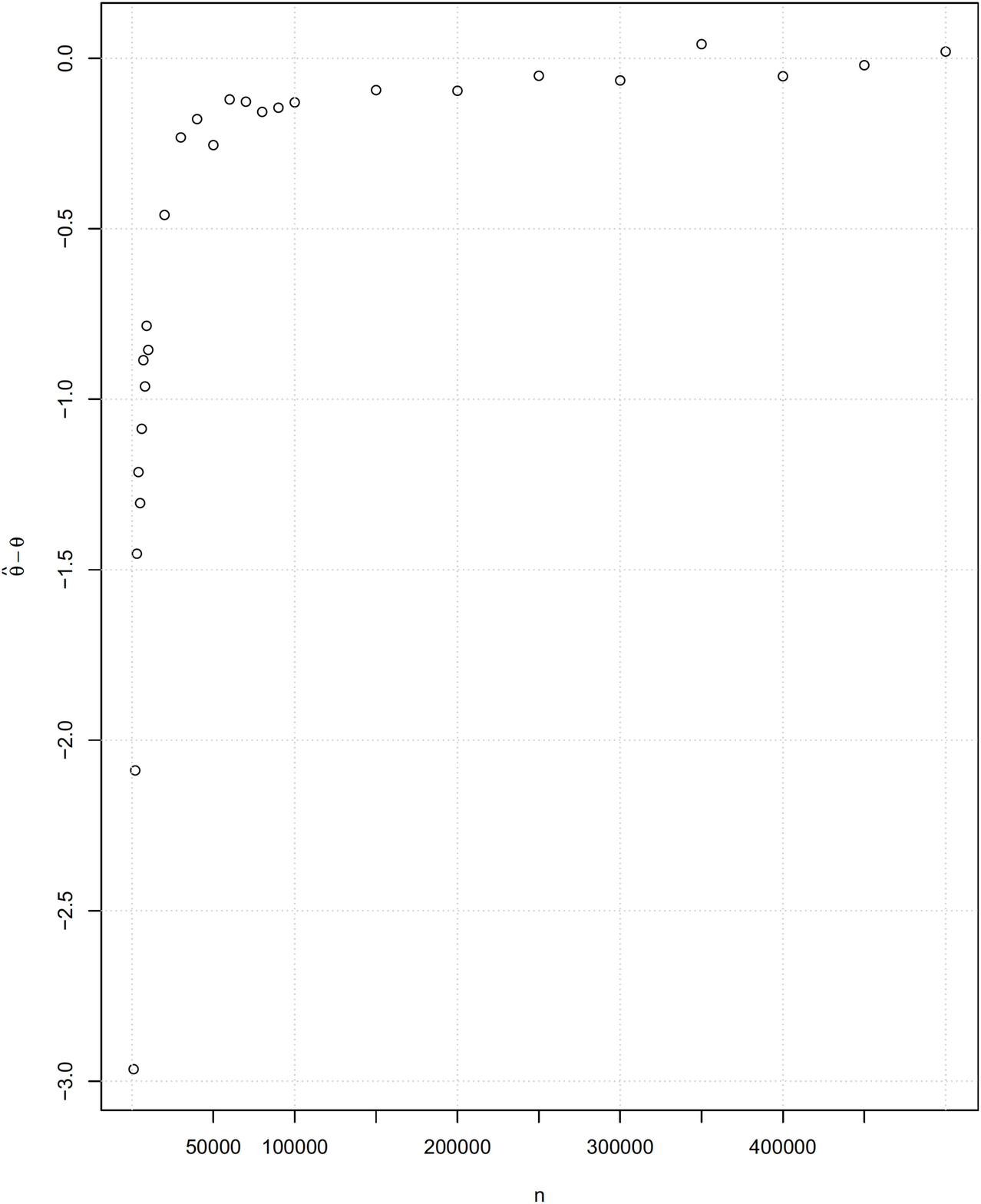}
\caption{Simulation result under a one-dimensional Vasicek model.}
\label{Vasicek}
\end{figure}

\begin{figure}[h]
\centering
\includegraphics[scale=0.07]{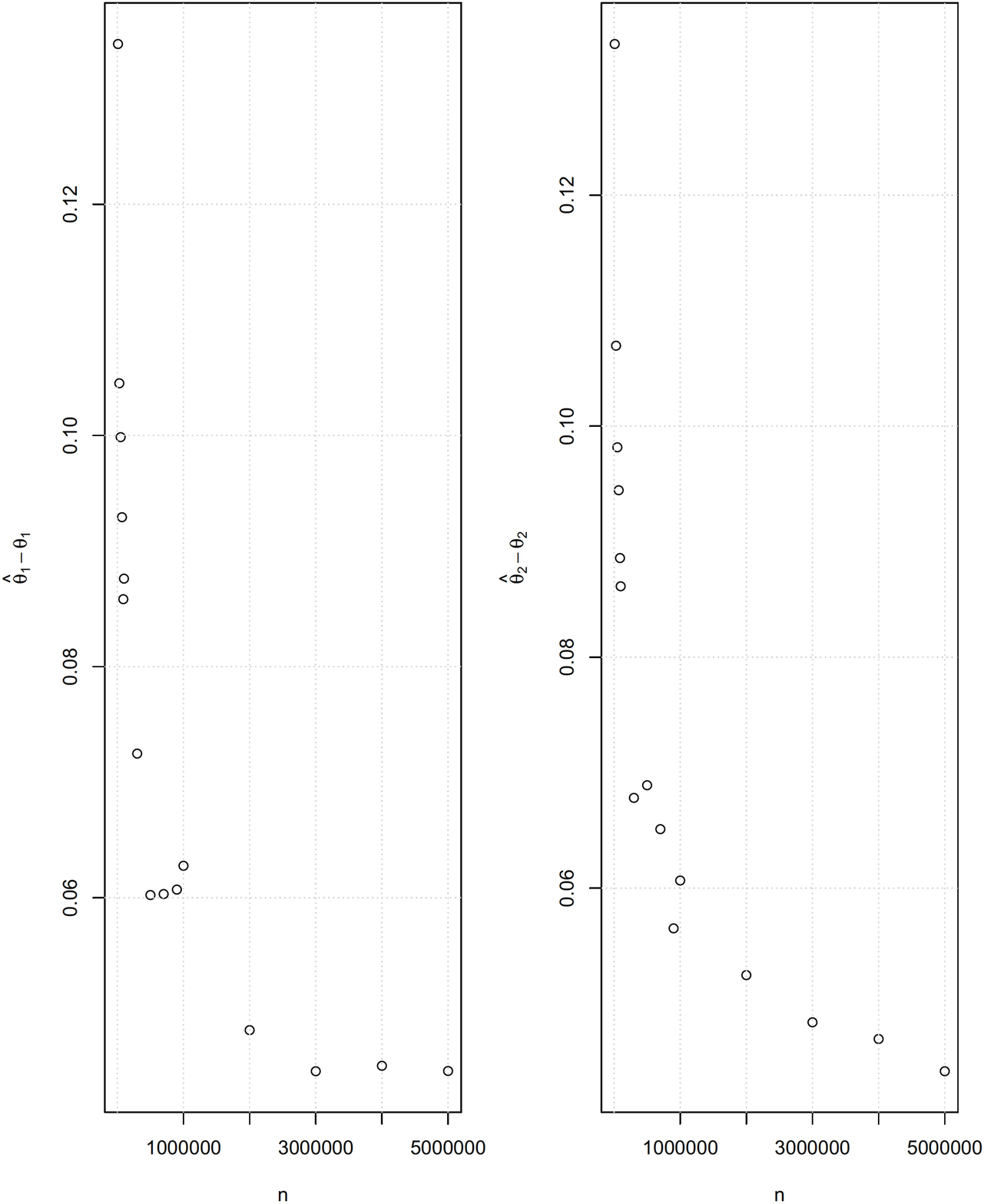}
\caption{Mean Absolute Error under a two-dimensional Heston model.}
\label{HestonMAE}
\end{figure}

Next, the two-dimensional case is simulated. The process $\nu_t$ evolves according to the Heston model, that is, $\nu_t=L(\beta-\nu_t)\td t + \sigma\sqrt{\nu_t}\td W_t$, with parameters $L=1, \beta=1.5$ and $\sigma=0.5$, and the initial value is $\nu_0=1.5$. We want to estimate $\theta^1$ and $\theta^2$ in equation (\ref{eq.Heston}), where $\sqrt{Z^{1,1}}=\sqrt{Z^{2,1}}=\sqrt{Z^{2,2}}=0.4$. $k$ and $l$ remains to be 6 and 13 respectively, and $\theta_0^1=\theta_0^2=5$.

The number of simulation times $n$ is set as 
\EQQ{\{10000,30000,\cdots,90000,100000,
300000,\cdots,900000,1000000, 2000000,\ldots,5000000\}.} 
For each $n$, we repeat the process by 500 times and calculate the Mean-Absolute-Error (MAE) of the estimators $\hat{\theta}$'s. Figure \ref{HestonMAE} shows the result. 


\section{Proofs}\label{se:proofs}

{\colord In this section, we prove the results in Section~\ref{se:2}.
In Section~\ref{sse:approx}, we introduce two functions $\check{H}_{n,\delta}$ and $\tilde{H}_{n,\delta}$ which are approximation of the quasi-log-likelihood $H_n$.
The function $\CHEH$ is introduced to control the event that either $Z_t$ or $\hat{Z}_l$ is close to degenerate for some $t$ or $l$,
and is equal to $H_n$ except on that event.
The function $\TILH$ is obtained by replacing the estimator $\hat{Z}_{l-1}$ in $\CHEH$ with $Z_{t_0^{l-1}}$.
In Section~\ref{consistency-proof-section}, we will show that the difference of $\PT^l \CHEH$ and $\PT^l\TILH$ can be asymptotically {\colora ignored},
and we consequently obtain consistency of $\hat{\theta}_n$.
To show Theorem~\ref{asymp-normal-thm}, we need {\colora an} accurate estimate for the difference of $\PT \CHEH(\theta_0)$ and $\PT \TILH(\theta_0)$, which is given in Proposition~\ref{score-diff-prop} of Section~\ref{sse:AN}. Together with asymptotic estimate Lemma~\ref{tildeH-conv-lemma} of $\PT^l\TILH$, we obtain {\colora then} the desired results.
}

\subsection{Approximation of $H_n$}\label{sse:approx}
For a vector $v$ and a matrix $A$, $[v]_i$ and $[A]_{ij}$ denote $(i,j)$ element of a matrix $A$ and $i$-th element of $v$, respectively.
For {\colorp $q>0$ and} a sequence $p_n$ of positive numbers, let us denote by $\{{\colorp \bar{R}_{n,q}}(p_n)\}_{n\in\mathbb{N}}$ and $\{{\colorp {\underbar{R}_{n,q}}}({\colora p_n})\}_{n\in\mathbb{N}}$
{\colora sequences} of random variables (which may also depend on $l$ and $\theta$) satisfying
\begin{equation}
\sup_{\theta,l}E[|p_n^{-1}{\colorp \bar{R}_{n,q}}(p_n)|^q]^{1/q}<\infty \quad {\rm and} \quad \sup_{\theta,l}E[|p_n^{-1}{\colorp \underbar{R}_{n,q}}(p_n)|^q]^{1/q}\to 0.
\end{equation}
{\colorp Then (A1-$p$) and (A2-$r$) imply 
\EQU{\label{DeltaY-eq} \Delta_l Y=\int_{t_0^l}^{t_0^{l-1}} \psi(X_t,Y_t,V_tV_t^\intercal)dt+\int_{t_0^l}^{t_0^{l-1}}V_tdW_t=\bar{R}_{n,p/r}(\sqrt{c_nh_n}).}
}

Let $Z_t=V_tV_t^\intercal$.
We first {\colord introduce a family of stopping times controlling the degeneracy of $Z_t$ and $\hat{Z}_l$.}
For any $\delta>0$, let
  $$T_{n,\delta}=\inf\{t^{l+1}_0; 0\leq l\leq L_n-1, {\colord \hat{Z}_l\not\in \mathfrak{P}_\delta \ {\rm or} \ Z_t\not\in \mathfrak{P}_\delta} \ {\rm for \ some} \  t\in[0,t_0^{l+1}]\},$$
{\colord where $\inf\emptyset=\infty$.}
Under (A1-$p$), $t_0^l<T_{n,\delta}$ implies that $\det \hat{Z}_{l-1}\geq \delta^{d_Y}$ and $\det Z_t\geq \delta^{d_Y}$ for $t\in [0,t_0^l]$
because $V$ has a continuous path.

{\colord Let} $\tilde{\psi}_l(\theta)=\psi(X_{t^l_0},Y_{t_0^l},Z_{t^{l-1}_0},\theta)$, and let
  $$\check{H}_{n,\delta}(\theta)=-\frac{1}{2}\sum_{l=1}^{L_n-1}{\colord (\Delta_l Y-c_nh_n\hat{\psi}_l(\theta))^\intercal\frac{\hat{Z}_{l-1}^{-1}}{c_nh_n}(\Delta_l Y-c_nh_n\hat{\psi}_l(\theta))}1_{\{t^l_0< T_{n,\delta}\}},$$
and
  $$\tilde{H}_{n,\delta}(\theta)=-\frac{1}{2}\sum_{l=1}^{L_n-1}{\colord (\Delta_l Y-c_nh_n\tilde{\psi}_l(\theta))^\intercal\frac{Z_{t^{l-1}_0}^{-1}}{c_nh_n}(\Delta_l Y-c_nh_n\tilde{\psi}_l(\theta))}1_{\{t^l_0< T_{n,\delta}\}}.$$

When $\delta$ is sufficiently small {\colora and $n$ sufficiently large}, $\check{H}_n$ corresponds to $H_n$
with high probability (see (\ref{H-checkH-diff})).
$\tilde{H}_n$ is an approximation of $\check{H}_n$ which is useful
when we deduce the asymptotic behavior.

{\colora The Burkholder-Davis-Gundy inequality and Jensen's inequality yield
\EQQ{E\bigg[\bigg|\int_{t_{m-1}^l}^{t_m^l}V_tdW_t\bigg|^{2p}\bigg]
\leq C_pE\bigg[\bigg(\int_{t_{m-1}^l}^{t_m^l}|V_t|^2dt\bigg)^p\bigg]
\leq C_ph_n^{p-1}E\bigg[\int_{t_{m-1}^l}^{t_m^l}|V_t|^{2p}dt\bigg]
\leq C_ph_n^p\sup_tE[|V_t|^{2p}],
}
which implies that $\Psi_{1,l,m}:=\int_{t_{m-1}^l}^{t_m^l}V_tdW_t=\bar{R}_{n,2p}(\sqrt{h_n})$ by (A1-$p$).
Similarly, (A1-$p$) and (A2-$r$) yield $\Psi_{2,l,m}:=\int_{t_{m-1}^l}^{t_m^l}\psi(X_t,Y_t,Z_t,\theta_0)dt=\bar{R}_{n,p/r}(h_n)$. Then by} It\^o's formula and the Cauchy-Schwarz inequality, (A1-$p$), and (A2-$r$) yield
\begin{eqnarray}
\hat{Z}_l&=&\frac{1}{c_nh_n}\sum_{m=1}^{c_n}(Y_{t^l_m}-Y_{t^l_{m-1}})(Y_{t^l_m}-Y_{t^l_{m-1}})^\intercal \nonumber \\
&=&{\colora \frac{1}{c_nh_n}\sum_{m=1}^{c_n}\bigg\{\int^{t_m^l}_{t^l_{m-1}}Z_tdt+2\mfa_{l,m}+\Psi_{2,l,m}\Psi_{1,l,m}^\intercal+(\Psi_{1,l,m}+\Psi_{2,l,m})\Psi_{2,l,m}^\intercal\bigg\}} \nonumber \\
&=&\frac{1}{c_nh_n}\sum_{m=1}^{c_n}\bigg\{\int^{t_m^l}_{t^l_{m-1}}Z_tdt+2\mfa_{l,m}+{\colorp \bar{R}_{n,\frac{p}{2r}}}(h_n^{3/2})\bigg\} \nonumber \\
&=&Z_{t_0^l}+\frac{2}{c_nh_n}\sum_{m=1}^{c_n}\mfa_{l,m}+{\colorp \bar{R}_{n,\frac{p}{2r}}}(\sqrt{c_nh_n}) \label{hatZ-eq1} \\
&=&Z_{t_0^l}+{\colorp \bar{R}_{n,\frac{p}{2r}}}(c_n^{-1/2}+\sqrt{c_nh_n}), \label{hatZ-eq2}
\end{eqnarray}
where 
\EQQ{[\mfa_{l,m}]_{ij}=\frac{1}{2}\sum_k\int^{t_m^l}_{t^l_{m-1}}([Y_t-Y_{t^l_{m-1}}]_i[V_t]_{jk}+[Y_t-Y_{t^l_{m-1}}]_j[V_t]_{ik})d[W_t]_k.}
\begin{discuss}
{\colorr $\sup_tE[|Z_t|^p]<\infty$ is required.
$$\frac{1}{ch_n}\sum_{m=1}^{c_n}\int (Z_t-Z_{t_0^l})dt=O\bigg(\frac{1}{c_nh_n}c_nh_n\cdot\sqrt{c_nh_n}\bigg).$$
}
\end{discuss}
Therefore, for any $\delta>0$ and {\colorp $q=p/(2r)$}, we obtain
\begin{equation}\label{hatZ-Z-diff-ineq}
P(\max_l|\hat{Z}_l-Z_{t_0^l}|>\delta)\leq \delta^{-q}\sum_lE[|\hat{Z}_l-Z_{t^l_0}|^q]=O(L_n(c_n^{-1/2}+\sqrt{c_nh_n})^q)\to 0,
\end{equation}
as $n\to \infty$ {\colorp if $q>1/\epsilon$}.
\begin{discuss}
{\colorr It is required that there exists $\epsilon>0$ such that $c_nh_nn^\epsilon\to 0$ and $n^\epsilon c_n^{-1}\to 0$.}
\end{discuss}

Then (A1-$p$) yields 
  \EQU{\label{Tn-lim} {\colord \lim_{\delta\to 0}\liminf_{n\to\infty}P(T_{n,\delta}=+\infty)=1,}}
and therefore, we have
  \EQU{\label{H-checkH-diff} \lim_{\delta\to 0}\liminf_{n\to\infty}P(\check{H}_{n,\delta}(\theta)=H_n(\theta) \ {\rm for \ any} \ \theta)=1. }

{\colord Equation (\ref{H-checkH-diff}) implies that the asymptotic behavior of $H_n$ is essentially the same as more tractable $\CHEH$ for sufficiently small $\delta>0$.
We further show that $\CHEH$ is asymptotically equivalent to $\TILH$ in Lemma~\ref{checkH-tildeH-diff-lemma} of the following section.}

\subsection{Proof of consistency}\label{consistency-proof-section}

\begin{lemma}\label{checkH-tildeH-diff-lemma}
{\colorp Let $p,r\geq 1$ such that (\ref{pr-cond}) is satisfied.} Assume (A1-$p$) and (A2-$r$). Then
\begin{equation}\label{Hn-lim}
(nh_n)^{-1}\sup_\theta|\check{H}_{n,\delta}(\theta)-\check{H}_{n,\delta}(\theta_0)-\tilde{H}_{n,\delta}(\theta)+\tilde{H}_{n,\delta}(\theta_0)|\overset{P}\to 0,
\end{equation}
as $n\to \infty$ for any $\delta>0$.
\end{lemma}

\begin{proof}
By the definitions of $\check{H}_n$ and $\tilde{H}_n$,
we can decompose the difference as 
\begin{eqnarray}
&&\check{H}_{n,\delta}(\theta)-\check{H}_{n,\delta}(\theta_0)-\tilde{H}_{n,\delta}(\theta)+\tilde{H}_{n,\delta}(\theta_0) \nonumber \\
&&\quad =-\frac{c_nh_n}{2}\sum_{l=1}^{L_n-1}\bigg(\tilde{\psi}_l(\theta)^\intercal(\hat{Z}_{l-1}^{-1}-Z_{t^{l-1}_0}^{-1})\tilde{\psi}_l(\theta)
-\tilde{\psi}_l(\theta_0)^\intercal(\hat{Z}_{l-1}^{-1}-Z_{t^{l-1}_0}^{-1})\tilde{\psi}_l(\theta_0)
\bigg)1_{\{t_0^l< T_{n,\delta}\}} \nonumber \\
&&\quad \quad +\sum_{l=1}^{L_n-1}\Delta_lY^\intercal(\hat{Z}_{l-1}^{-1}-Z_{t^{l-1}_0}^{-1})(\tilde{\psi}_l(\theta)-\tilde{\psi}_l(\theta_0))
1_{\{t_0^l< T_{n,\delta}\}} \nonumber \\
&&\quad \quad -\frac{c_nh_n}{2}\sum_{l=1}^{L_n-1}\bigg(\hat{\psi}_l(\theta)^\intercal\hat{Z}_{l-1}^{-1}\hat{\psi}_l(\theta)
-\hat{\psi}_l(\theta_0)^\intercal\hat{Z}_{l-1}^{-1}\hat{\psi}_l(\theta_0) \nonumber \\
&&\quad \quad \quad \quad \quad \quad \quad \quad 
-\tilde{\psi}_l(\theta)^\intercal\hat{Z}_{l-1}^{-1}\tilde{\psi}_l(\theta)+\tilde{\psi}_l(\theta_0)^\intercal\hat{Z}_{l-1}^{-1}\tilde{\psi}_l(\theta_0)\bigg)1_{\{t_0^l< T_{n,\delta}\}} \nonumber \\
&&\quad \quad +\sum_{l=1}^{L_n-1}\Delta_l Y^\intercal\hat{Z}_{l-1}^{-1}(\hat{\psi}_l(\theta)-\hat{\psi}_l(\theta_0)-\tilde{\psi}_l(\theta)+\tilde{\psi}_l(\theta_0))1_{\{t_0^l< T_{n,\delta}\}} \nonumber \\
&&\quad =:\Lambda_1(\theta)+\Lambda_2(\theta)+\Lambda_3(\theta)+\Lambda_4(\theta). \nonumber
\end{eqnarray}

Then it is sufficient to show that
{\colorp $\sup_\theta |\Lambda_j(\theta)|=\underbar{R}_{n,pr'/4}(nh_n)$
for $1\leq j\leq 4$, where $r'=1/r$.}
\begin{discuss}
{\colorr $$E[\sup_\theta |\Lambda_1|^p]\leq C(c_nh_nL_n)^p\frac{1}{L_n}\sum_lE\bigg[\sup_\theta\bigg|\tilde{\psi}_l(\theta)^\intercal(\hat{Z}_{l-1}^{-1}-Z_{t^{l-1}_0}^{-1})\tilde{\psi}_l(\theta)
-\tilde{\psi}_l(\theta_0)^\intercal(\hat{Z}_{l-1}^{-1}-Z_{t^{l-1}_0}^{-1})\tilde{\psi}_l(\theta_0)1_{\{t_0^l< T_{n,\delta}\}}\bigg|^p\bigg].$$}
\end{discuss}
{\colora (A2-$r$) yields 
\EQNN{\sup_\theta |\hat{\psi}_l(\theta)\hat{\psi}_l(\theta)^\intercal-\tilde{\psi}_l(\theta)\tilde{\psi}_l(\theta)^\intercal| 
&=\sup_\theta |\hat{\psi}_l(\theta)(\hat{\psi}_l(\theta)-\tilde{\psi}_l(\theta))^\intercal+(\hat{\psi}_l(\theta)-\tilde{\psi}_l(\theta))\tilde{\psi}_l(\theta)^\intercal| \\
&\leq CC_\delta (1+|X_{t_0^l}|+|Y_{t_0^l}|+|\hat{Z}_{l-1}|+|Z_{t_0^{l-1}}|)^{2r}|\hat{Z}_{l-1}-Z_{t_0^{l-1}}|,}
on $\{t_0^l<T_{n,\delta}\}$ for any $\delta>0$.
Then (A1-$p$), (\ref{hatZ-eq2}), and the Cauchy-Schwartz inequality yield
$\sup_\theta|\hat{\psi}_l(\theta)\hat{\psi}_l(\theta)^\intercal-\tilde{\psi}_l(\theta)\tilde{\psi}_l(\theta)^\intercal|1_{\{t_0^l<T_{n,\delta}\}}= \underbar{R}_{n,pr'/4}(1)$.
We also have
\EQQ{(\hat{Z}_{l-1}^{-1}-Z_{t^{l-1}_0}^{-1})1_{\{t_0^l<T_{n,\delta}\}}=\hat{Z}_{l-1}^{-1}(Z_{t^{l-1}_0}-\hat{Z}_{l-1})Z_{t^{l-1}_0}^{-1}1_{\{t_0^l<T_{n,\delta}\}}=\underbar{R}_{n,pr'/2}(1),}
and hence we obtain}
{\colorp $$\sup_\theta|\Lambda_j(\theta)|=\underbar{R}_{n,pr'/4}(c_nh_nL_n)=\underbar{R}_{n,pr'/4}(nh_n),$$}
for $j\in \{1,3\}$.

Moreover, since 
\begin{eqnarray}
\Delta_l Y&=&\int^{t^{l+1}_0}_{t_0^l}V_sdW_s+\tilde{\psi}_l(\theta_0)(t^{l+1}_0-t^l_0)+{\colorp \bar{R}_{n,pr'/3}}((c_nh_n)^{3/2}) \label{delY-eq2} \\
&=&V_{t^l_0}\Delta_lW+{\colorp \bar{R}_{n,pr'/3}}(c_nh_n), \label{delY-eq}
\end{eqnarray}
{\colorp by (\ref{DeltaY-eq}) and (\ref{hatZ-eq2}),} we have {\colora
\EQQ{\PT^l\Lambda_2
=\sum_{l=1}^{L_n-1}\Delta_lW^\intercal V_{t_0^l}^\intercal(\hat{Z}_{l-1}^{-1}-Z_{t^{l-1}_0}^{-1})\PT^l(\tilde{\psi}_l(\theta)-\tilde{\psi}_l(\theta_0))
1_{\{t_0^l< T_{n,\delta}\}}+\underbar{R}_{n,pr'/4}(nh_n),}
for $l\in \{0,1\}$. The Burkholder-Davis-Gundy inequality and the triangle inequality yield
\EQNN{&E\bigg[\bigg|\sum_{l=1}^{L_n-1}\Delta_lW^\intercal V_{t_0^l}^\intercal(\hat{Z}_{l-1}^{-1}-Z_{t^{l-1}_0}^{-1})\PT^l(\tilde{\psi}_l(\theta)-\tilde{\psi}_l(\theta_0))1_{\{t_0^l< T_{n,\delta}\}}\bigg|^q\bigg] \\
&\quad \leq C_q\bigg(\sum_{l=1}^{L_n-1}E[|\Delta_lW^\intercal V_{t_0^l}^\intercal(\hat{Z}_{l-1}^{-1}-Z_{t^{l-1}_0}^{-1})\PT^l(\tilde{\psi}_l(\theta)-\tilde{\psi}_l(\theta_0))1_{\{t_0^l< T_{n,\delta}\}}|^q]^{2/q}\bigg)^{q/2},
}
for $q\geq 2$. Then we obtain
\EQU{\label{Lambda2-est}  \PT^l\Lambda_2 =\bar{R}_{n,p/(r+1)}(\sqrt{L_nc_nh_n})+\underbar{R}_{n,pr'/4}(nh_n)=\underbar{R}_{n,pr'/4}(nh_n),}
for $l\in \{0,1\}$, and similarly we have}
\begin{eqnarray}\label{Lambda4-est}
\PT^l\Lambda_4&=&\sum_{l=1}^{L_n-1}\Delta_l W^\intercal V_{t_0^l}^\intercal\hat{Z}_{l-1}^{-1}\PT^l(\hat{\psi}_l(\theta)-\hat{\psi}_l(\theta_0)-\tilde{\psi}_l(\theta)+\tilde{\psi}_l(\theta_0))1_{\{t_0^l< T_{n,\delta}\}}
+{\colorp \underbar{R}_{n,pr'/4}}(nh_n) \nonumber \\
&=&{\colorp \bar{R}_{n,p/(r+1)}}(\sqrt{L_nc_nh_n})+{\colorp \underbar{R}_{n,pr'/4}}(nh_n)={\colorp \underbar{R}_{n,pr'/4}}(nh_n),
\end{eqnarray}
for $l\in \{0,1\}$.

Sobolev's inequality (\ref{sobolev}) yields $\sup_\theta|\Lambda_j(\theta)|={\colorp \underbar{R}_{n,pr'/4}}(nh_n)$ for $j\in \{2,4\}$, which completes the proof.

\end{proof}

\noindent
{\bf Proof of Theorem~\ref{consistency-thm}.}

{\colord We first deduce the limit of $(nh_n)^{-1}(H_n(\theta)-H_n(\theta_0))$.}
(\ref{delY-eq2}) yields
\EQN{\label{tildeH-eq} &\tilde{H}_{n,\delta}(\theta)-\tilde{H}_{n,\delta}(\theta_0) \\
&\quad =-\frac{1}{2}\sum_{l=1}^{L_n-1}\bigg(c_nh_n(\tilde{\psi_l}(\theta)^\intercal Z_{t_0^{l-1}}^{-1}\tilde{\psi}_l(\theta) - \tilde{\psi_l}(\theta_0)^\intercal Z_{t_0^{l-1}}^{-1}\tilde{\psi}_l(\theta_0))
-2\Delta_lY^\intercal Z_{t^{l-1}_0}^{-1}(\tilde{\psi}_l(\theta)-\tilde{\psi}_l(\theta_0))\bigg)1_{\{t_0^l<T_{n,\delta}\}} \\
&\quad =-\frac{c_nh_n}{2}\sum_{l=1}^{L_n-1}\bigg(\tilde{\psi_l}(\theta)^\intercal Z_{t_0^{l-1}}^{-1}\tilde{\psi}_l(\theta) - \tilde{\psi_l}(\theta_0)^\intercal Z_{t_0^{l-1}}^{-1}\tilde{\psi}_l(\theta_0)
-2\tilde{\psi}_l(\theta_0)^\intercal Z_{t^{l-1}_0}^{-1}(\tilde{\psi}_l(\theta)-\tilde{\psi}_l(\theta_0))\bigg)1_{\{t_0^l<T_{n,\delta}\}} \\
&\quad \quad +{\colorp \bar{R}_{n,pr'/4}}(nh_n\sqrt{c_nh_n}+\sqrt{nh_n}) \\
&\quad =\mcy_n(\theta) + {\colorp \underbar{R}_{n,pr'/4}}(nh_n),
}
{\colord on $\{t_0^{L_n-1}<T_{n,\delta}\}$,} where 
\EQQ{\mcy_n(\theta)=-\int_0^{nh_n}(\psi(X_t,Y_t,Z_t,\theta)-\psi(X_t,Y_t,Z_t,\theta_0))^\intercal Z_t^{-1}(\psi(X_t,Y_t,Z_t,\theta)-\psi(X_t,Y_t,Z_t,\theta_0))dt.}

Similarly, we have
$\PT\tilde{H}_{n,\delta}(\theta)=\PT\mcy_n(\theta)+ {\colorp \underbar{R}_{n,pr'/4}}(nh_n)$
on {\colord $\{t_0^{L_n-1}<T_{n,\delta}\}$}, and hence Sobolev's inequality yields
\EQU{\label{H-lim} \sup_\theta|\TILH(\theta){\colord -\TILH(\theta_0)}-\mcy_n(\theta)|={\colorp \underbar{R}_{n,pr'/4}}(nh_n).}

Let
  \EQQ{\mcy(\theta)=\DIV{l}{
   -\int(\psi(x,z,\theta)-\psi(x,z,\theta_0))^\intercal z^{-1}(\psi(x,z,\theta)-\psi(x,z,\theta_0))\pi(dxdz) \\
   \hspace{6cm} {\rm if~Point~1~of~{\colora (A3\mathchar`-p)}~is~satisfied} \\
   -\int(\psi(x,y,z,\theta)-\psi(x,y,z,\theta_0))^\intercal z^{-1}(\psi(x,y,z,\theta)-\psi(x,y,z,\theta_0))\pi({\colord dxdydz}) \\
   \hspace{6cm} {\rm if~Point~2~of~{\colora (A3\mathchar`-p)}~is~satisfied}. \\
  }}

Then for any $\epsilon,\eta>0$, (\ref{H-lim}), {\colord (\ref{Tn-lim})} and {\colora (A3-$p$)} yield
  \EQNN{&P(|(nh_n)^{-1}(H_n(\theta)-H_n(\theta_0))-\mcy(\theta)|>\eta) \\
   &\quad \leq P(|(nh_n)^{-1}\mcy_n(\theta)-\mcy(\theta)|>\eta/2)+P(t_0^{L_n-1}\geq T_{n,\delta}) \\
   &\quad \quad +P(|(nh_n)^{-1}(\CHEH(\theta)-\CHEH(\theta_0)-\mcy_n(\theta))|>\eta/2, t_0^{L_n-1}<T_{n,\delta}) \\
   &\quad <\epsilon,
  }
for $\theta\in\Theta$, sufficiently large $n$, and sufficiently small $\delta$.
Then we have
\EQU{\label{Hn-lim-eq} (nh_n)^{-1}(H_n(\theta)-H_n(\theta_0))\overset{P}\to \mcy(\theta),}
as $n\to\infty$ for any $\theta\in \Theta$.

{\colord Next, we show that consistency of $\hat{\theta}_n$ is obtained if (\ref{Hn-lim-eq}) holds uniformly in $\theta$.}

{\colord (A2-$r$), (A3-$p$), and (A4)} imply that $\mathcal{Y}(\theta)$ is continuous on $\theta$ and
\begin{equation*}
\mathcal{Y}(\theta){\colord =0} \quad \Longrightarrow \quad \theta=\theta_0.
\end{equation*}
\begin{discuss}
{\colorr \EQNN{\mcy(\theta)=0 &\Longrightarrow \psi(x,y,z,\theta)=\psi(x,y,z,\theta_0) \quad \pi\mathchar`-{\rm a.s.} \\
  &\Longrightarrow \psi(x,y,z,\theta)=\psi(x,y,z,\theta_0) \ {\rm on} \ \{(x,y,z)|\pi(x,y,z)>0\} \ dxdz\mathchar`-{\rm a.e.} \\
&\Longrightarrow \psi(x,y,z,\theta)=\psi(x,y,z,\theta_0) \ {\rm on} \ \bar{\{(x,y,z)|\pi(x,y,z)>0\}},}
by continuity of $\psi$.
}
\end{discuss}

Then for any $\epsilon,\delta>0$, there exists $\eta>0$ such that
\EQU{\label{mcy-ineq} P\bigg(\inf_{|\theta-\theta_0|\geq \delta}(-\mcy(\theta))<\eta\bigg)<\frac{\epsilon}{2},}
\begin{discuss}
{\colorr $\inf(-\mcy)=0$ implies that an accumulation point $\theta'$ satisfies $\mcy(\theta')=0$. 
So $P(\inf(-\mcy)>0)=1$.}
\end{discuss}
Because $H_n(\hat{\theta}_n)-H_n(\theta_0)\geq 0$ by the definition {\colord of $\hat{\theta}_n$}, 
together with (\ref{mcy-ineq}), we have
$$P(|\hat{\theta}_n-\theta_0|\geq \delta){\colord <} P\bigg(\sup_\theta\bigg|\frac{1}{nh_n}(H_n(\theta)-H_n(\theta_0))-\mcy(\theta)\bigg|>\eta\bigg)+\frac{\epsilon}{2},$$
for any $n$.

Then it is sufficient to show 
\EQU{\label{sup-Hn-lim} \sup_{\theta\in \Theta}\bigg|\frac{1}{nh_n}(H_n(\theta)-H_n(\theta_0))-\mcy(\theta)\bigg|\overset{P}\to 0,}
as $n\to\infty$.
By (\ref{Hn-lim-eq}), it is sufficient to show C-tightness of 
$(nh_n)^{-1}(H_n(\cdot)-H_n(\theta_0))$.

{\colord Finally, we show C-tightness.}
Similarly to the proof of Lemma~\ref{checkH-tildeH-diff-lemma},
we obtain ${\colord \sup_\theta|\PT^l\Lambda_j(\theta)|}={\colorp \underbar{R}_{n,pr'/4}}(nh_n)$ for $1\leq j\leq 4$ and $l\in \{1,2\}$. Together with a similar argument to (\ref{tildeH-eq}),
for {\colorp $q=pr'/4$} and $l\in\{1,2\}$, there exists $N'\in\mbbn$ such that
\EQNN{&\sup_{n\geq N',\theta}E[|(nh_n)^{-1}\PT^l\check{H}_{n,\delta}(\theta)|^q1_{\{t_0^{L_n-1}<T_{n,\delta}\}}] \\
&\quad \leq \sup_{n\geq N',\theta}E[|(nh_n)^{-1}\PT^l\tilde{H}_{n,\delta}(\theta)|^q1_{\{t_0^{L_n-1}<T_{n,\delta}\}}]+1 \\
&\quad \leq \sup_{n\geq N',\theta}E[|(nh_n)^{-1}\PT^l\mcy_n(\theta)|^q1_{\{t_0^{L_n-1}<T_{n,\delta}\}}]+2<\infty.}

Then Sobolev's inequality yields
\EQQ{\limsup_{n\to\infty}E[\sup_\theta |(nh_n)^{-1}\PT\CHEH(\theta)|^q1_{\{t_0^{L_n-1}<T_{n,\delta}\}}]<\infty.}
Together with {\colord (\ref{Tn-lim}) and} (\ref{H-checkH-diff}), for any $\epsilon>0$, there exists $K>0$ and $N''\in\mbbn$ such that
  \EQQ{\sup_{n\geq N''}P(\sup_\theta |(nh_n)^{-1}\PT H_n(\theta)|>K)<\epsilon.}
\begin{discuss}
{\colorr For any $\delta>0$, 
  \EQNN{&P(\sup_\theta|(nh_n)^{-1}\PT H_n(\theta)|>K) \\
   &\quad \leq P(H_n\neq \CHEH \ {\rm for \ some} \ \theta)
   +P(t_0^{L_n-1}\geq T_{n,\delta})
   +P(\sup_\theta|(nh_n)^{-1}\PT \CHEH(\theta)|1_{\{t_0^{L_n-1}< T_{n,\delta}\}}>K).}
}
\end{discuss}

Then C-tightness condition {\colord (Theorem 7.3)} in Billingsley~(1999) yields {\colord the desired result}.

\qed

\subsection{The proof of asymptotic normality}\label{sse:AN}

{\colord We show asymptotic normality of $\hat{\theta}_n$ in this section.
For this purpose, we show a stronger estimate of $\PT \CHEH(\theta_0)-\PT\TILH(\theta_0)$ in Proposition~\ref{score-diff-prop}.
We first prepare a fundamental result which is repeatedly used in the following.}

\begin{lemma}\label{F-sum-lemma}
Let $(F_l)_{l=1}^{L_n-1}$ be random variables satisfying that $F_l$ is $\F_{t_0^l}$-measurable and that $E[F_l|\F_{t_0^{l-1}}]=0$ for $1\leq l\leq L_n-1$. Then
\EQQ{E\bigg[\bigg|\sum_{l=1}^{L_n-1}F_l1_{\{t_0^l<T_{n,\delta}\}}\bigg|^2\bigg]\leq 4\sum_{l=1}^{L_n-1}E[F_l^21_{\{t_0^{l-1}<T_{n,\delta}\}}].}
\end{lemma}
\begin{proof}
Because $E[F_l1_{\{t_0^{l-1}<T_{n,\delta}\}}|\F_{t_0^{l-1}}]=0$
and $\sum_{l=1}^{L_n-1}1_{\{t_0^{l-1}<T_{n,\delta}\leq {\colora t_0^l}\}}\leq 1$, 
the Cauchy-Schwarz inequality yields
\EQNN{&E\bigg[\bigg|\sum_{l=1}^{L_n-1}F_l1_{\{t_0^l<T_{n,\delta}\}}\bigg|^2\bigg] \\
&\quad \leq 2E\bigg[\bigg|\sum_{l=1}^{L_n-1}F_l1_{\{t_0^{l-1}<T_{n,\delta}\}}\bigg|^2\bigg]+2E\bigg[\bigg|\sum_{l=1}^{L_n-1}F_l1_{\{t_0^{l-1}<T_{n,\delta}\leq {\colora t_0^l}\}}\bigg|^2\bigg] \\
&\quad \leq 2\sum_{l=1}^{L_n-1}E[F_l^21_{\{t_0^{l-1}<T_{n,\delta}\}}] +2\sum_{l=1}^{L_n-1}E[F_l^21_{\{t_0^{l-1}<T_{n,\delta}\}}]\sum_{l=1}^{L_n-1}E[1_{\{t_0^{l-1}<T_{n,\delta}\leq {\colora t_0^l}\}}] \\
&\quad \leq 4\sum_{l=1}^{L_n-1}E[F_l^21_{\{t_0^{l-1}<T_{n,\delta}\}}].
}
\end{proof}

\begin{proposition}\label{score-diff-prop}
{\colorp Let $p,r\geq 1$ such that (\ref{pr-cond}) is satisfied.}
Assume (A1-$p$), (A2$'$-$r$) and that $n^3h_n^5\to 0$. Then
\begin{equation*}
\frac{1}{\sqrt{nh_n}}\partial_\theta \check{H}_{n,\delta}(\theta_0)-\frac{1}{\sqrt{nh_n}}\partial_\theta \tilde{H}_{n,\delta}(\theta_0)\overset{P}\to 0,
\end{equation*}
as $n\to \infty$ for any $\delta>0$.
\end{proposition}

\begin{proof}
{\colora For a positive sequence $(c_n)_{n\in\mbbn}$ and random variables $(U_n)_{n\in\mbbn}$, we denote $U_n=O_P(c_n)$ if $(c_n^{-1}U_n)_{n\in\mbbn}$ is tight, and we denote $U_n=o_P(c_n)$ if $c_n^{-1}U_n\overset{P}\to 0$.}
First, (\ref{hatZ-eq1}) and (\ref{hatZ-eq2}) imply
\begin{eqnarray}\label{psi-diff-eq}
&&\partial_\theta \hat{\psi}_l\hat{\psi}_l^\intercal(\theta_0)-\partial_\theta \tilde{\psi}_l\tilde{\psi}_l^\intercal(\theta_0) \nonumber \\
&&\quad ={\colord \sum_{i,j}}\int^1_0{\colord \partial_{z_{ij}}}(\partial_\theta \psi\psi^\intercal)(X_{t^l_0},Y_{t_0^l},u\hat{Z}_{l-1}+(1-u)Z_{t^{l-1}_0},\theta_0)du\cdot {\colord [\hat{Z}_{l-1}-Z_{t^{l-1}_0}]_{ij}} \nonumber \\
&&\quad =2\sum_{i,j}\frac{\partial_{z_{ij}}(\partial_\theta \acute{\psi}_l \acute{\psi}_l^\intercal)(\theta_0)}{c_nh_n}
\sum_{m=1}^{c_n}[\mfa_{l-1,m}]_{ij}+{\colorp O_P}(c_n^{-1}+\sqrt{c_nh_n}) \nonumber \\
&&\quad ={\colorp O_P}(c_n^{-1/2}+\sqrt{c_nh_n}),
\end{eqnarray}
and similarly
  \EQN{\label{psi-diff-eq2} \PT \hat{\psi}_l(\theta_0)-\partial_\theta \tilde{\psi}_l(\theta_0)
   &=2{\colord \sum_{i,j}}\frac{{\colord \partial_{z_{ij}}}\partial_\theta \acute{\psi}_l(\theta_0)}{c_nh_n}\sum_{m=1}^{c_n}{\colord [\mfa_{l-1,m}]_{ij}}+{\colorp O_P}(c_n^{-1}+\sqrt{c_nh_n}) \\
   &={\colorp O_P}(c_n^{-1/2}+\sqrt{c_nh_n}),
  }
if $t_0^l<T_{n,\delta}$,
{\colord where} $\acute{\psi}_l(\theta)=\psi(X_{t_0^{l-1}},Y_{t_0^{l-1}},Z_{t_0^{l-1}},\theta)$. 
\begin{discuss}
{\colorr We automatically obtain $\{u\hat{Z}_{l-1}+(1-u)Z_{t_0^{l-1}}\}_{u\in [0,1]}\subset \mathfrak{P}$ because the summation of two p.d. matrices is p.d.}
\end{discuss}

Moreover, (\ref{hatZ-eq1}), (\ref{hatZ-eq2}) and (\ref{cn-condition}) yield
\begin{eqnarray}
\hat{Z}_l^{-1}-Z_{t^l_0}^{-1}
&=& Z_{t_0^l}^{-1}(Z_{t_0^l}-\hat{Z}_l)Z_{t_0^l}^{-1}+Z_{t_0^l}^{-1}(Z_{t_0^l}-\hat{Z}_l)\hat{Z}_l^{-1}(Z_{t_0^l}-\hat{Z}_l)Z_{t_0^l}^{-1} \nonumber \\
&=& -\frac{2}{c_nh_n}\sum_{m=1}^{c_n}Z_{t_0^l}^{-1}\mfa_{l,m}Z_{t_0^l}^{-1}+{\colora \underbar{R}_{n,pr'/4}}((nh_n)^{-1/2}) \label{invZ-diff-eq2} \\
&=&{\colora \bar{R}_{n,pr'/4}}(c_n^{-1/2}+(nh_n)^{-1/2}), \label{invZ-diff-eq}
\end{eqnarray}
on $\{t_0^l<T_{n,\delta}\}$.
\begin{discuss}
{\colorr $\bar{R}_n(\sqrt{c_nh_n}+c_n^{-1})=\underbar{R}_n((nh_n)^{-1/2})$.}
\end{discuss}

Then Lemma~\ref{F-sum-lemma} yields
\begin{eqnarray}\label{Lambda1-est}
&&\frac{1}{\sqrt{nh_n}}\partial_\theta \Lambda_1(\theta_0) \nonumber \\
&&\quad=-\frac{c_nh_n}{\sqrt{nh_n}}\sum_{l=1}^{L_n-1}\partial_\theta\tilde{\psi}_l(\theta_0)^\intercal(\hat{Z}_{l-1}^{-1}-Z_{t^{l-1}_0}^{-1})\tilde{\psi}_l(\theta_0)1_{\{t_0^l<T_{n,\delta}\}} \nonumber \\
&&\quad=\frac{2}{\sqrt{nh_n}}\sum_{l=1}^{L_n-1}\sum_{m=1}^{c_n}\partial_\theta\tilde{\psi}_l(\theta_0)^\intercal
Z_{t_0^{l-1}}^{-1}\mfa_{l-1,m}Z_{t_0^{l-1}}^{-1}\tilde{\psi}_l(\theta_0)1_{\{t_0^l<T_{n,\delta}\}}+{\colorp o_P}(1) \nonumber \\
&&\quad=\frac{2}{\sqrt{nh_n}}\sum_{l=1}^{L_n-1}\sum_{m=1}^{c_n}\PT\acute{\psi}_l(\theta_0)^\intercal
Z_{t_0^{l-1}}^{-1}\mfa_{l-1,m}Z_{t_0^{l-1}}^{-1}\acute{\psi}_l(\theta_0)1_{\{t_0^l<T_{n,\delta}\}}+{\colorp o_P}(1) \nonumber \\
&&\quad={\colorp O_P}\bigg(\frac{1}{\sqrt{nh_n}}\sqrt{L_n}\sqrt{c_n}h_n\bigg)+{\colorp o_P}(1)\overset{P}\to 0, 
\end{eqnarray}
\begin{discuss}
{\colorr $\PT^l\tilde{\psi}_l-\PT^l\acute{\psi}_l=O(\sqrt{c_nh_n})$. 
The residual when replacing $\tilde{\psi}$ with $\acute{\psi}$ is
\EQQ{O_P\bigg(\frac{1}{\sqrt{nh_n}}L_nc_n\sqrt{c_nh_n}h_n\bigg)
=O_P(\sqrt{nh_n^2c_n})\overset{P}\to 0.}
\EQQ{E[|\PT \acute{\psi} \cdots |^2]\leq C_pE[(1+|X_{t_0^{l-1}}|+|Y_{t_0^{l-1}}|+|Z_{t_0^{l-1}}|)^{8r+4}]}
Then it is bounded because $\frac{p}{4r}>\frac{2}{\epsilon}\geq 2$ implies $p\geq 8r \geq 4r+2$.

}
\end{discuss}
and (\ref{cn-condition}), (\ref{psi-diff-eq}), and (\ref{invZ-diff-eq2}) yield
\begin{eqnarray}
&&\frac{1}{\sqrt{nh_n}}\partial_\theta \Lambda_3(\theta_0) \nonumber \\
&&\quad = -\frac{c_nh_n}{\sqrt{nh_n}}\sum_{l=1}^{L_n-1}\bigg(\partial_\theta \hat{\psi}_l(\theta_0)^\intercal\hat{Z}_{l-1}^{-1}\hat{\psi}_l(\theta_0)
-\partial_\theta \tilde{\psi}_l(\theta_0)^\intercal\hat{Z}_{l-1}^{-1}\tilde{\psi}_l(\theta)\bigg)1_{\{t_0^l< T_{n,\delta}\}} \nonumber \\
&&\quad = -\frac{1}{\sqrt{nh_n}}\sum_{l=1}^{L_n-1}\sum_{m=1}^{c_n}{\rm tr}(\mathfrak{B}_{l,m}\hat{Z}_{l-1}^{-1})1_{\{t_0^l< T_{n,\delta}\}}+{\colorp o_P}(1) \nonumber \\
&&\quad = -\frac{1}{\sqrt{nh_n}}\sum_{l=1}^{L_n-1}\sum_{m=1}^{c_n}{\rm tr}(\mathfrak{B}_{l,m}Z_{t_0^{l-1}}^{-1})1_{\{t_0^l< T_{n,\delta}\}} \nonumber \\
&&\quad \quad +\frac{2}{c_nh_n\sqrt{nh_n}}\sum_{l=1}^{L_n-1}\sum_{m,m'=1}^{c_n}{\rm tr}(\mathfrak{B}_{l,m}Z_{t_0^{l-1}}^{-1}\mfa_{l-1,m'}Z_{t_0^{l-1}}^{-1})1_{\{t_0^l< T_{n,\delta}\}}+{\colorp o_P}(1), \nonumber
\end{eqnarray}
where $\mathfrak{B}_{l,m}={\colora 2}\sum_{i,j}\partial_{z_{ij}}(\partial_\theta \acute{\psi}_l \acute{\psi}_l^\intercal)(\theta_0)[\mfa_{l-1,m}]_{ij}$.
\begin{discuss}
{\colorr The residual term:
\begin{equation*}
o_P\bigg(\frac{c_nL_n}{\sqrt{nh_n}}h_n(nh_n)^{-1/2}\bigg)=o_P(1).
\end{equation*}
}
\end{discuss}

The first term in the right-hand side of the above equation is equal to
\EQQ{{\colorp O_P}((nh_n)^{-1/2}\sqrt{L_nc_n}h_n)\overset{P}\to 0,}
by Lemma~\ref{F-sum-lemma}.
The second term is equal to
\EQQ{{\colorp O_P}\bigg(\frac{L_nc_nh_n^2}{c_nh_n\sqrt{nh_n}}\bigg)
+{\colorp O_P}\bigg(\frac{\sqrt{L_nc_n^2}h_n^2}{c_nh_n\sqrt{nh_n}}\bigg)\overset{P}\to 0,}
by Lemma~\ref{F-sum-lemma}, (\ref{cn-condition}), and
$E[\sum_{m\neq m'}{\rm tr}(\mathfrak{B}_{l,m}Z_{t_0^{l-1}}^{-1}\mfa_{l-1,m'}Z_{t_0^{l-1}}^{-1})|\F_{t_0^{l-1}}]=0$.
Therefore, we have
\EQU{\label{Lambda3-est} (nh_n)^{-1/2}\PT\Lambda_3(\theta_0)\overset{P}\to 0,}
as $n\to\infty$.
\begin{discuss}
{\colorr 
\EQQ{E[{\rm tr}(\mathfrak{B}_{l,m}Z_{t_0^{l-1}}^{-1})^21_{\{t_0^l< T_{n,\delta}\}}]
	\leq CE[(1+|X|+|Y|+|Z|)^{8r+4}],}
\EQQ{E[{\rm tr}(\mathfrak{B}_{l,m}Z_{t_0^{l-1}}^{-1}\mfa_{l-1,m'}Z_{t_0^{l-1}}^{-1})^2]
	\leq CE[(1+|X|+|Y|+|Z|)^{12r+4}].}
}
\end{discuss}

Moreover, (\ref{invZ-diff-eq}), (\ref{delY-eq2}), {\colorp (\ref{invZ-diff-eq2})} Lemma~\ref{F-sum-lemma}, and {\colora the parallelogram law} yield
\begin{eqnarray}\label{Lambda2-est2}
&&E\bigg[\bigg|\frac{1}{\sqrt{nh_n}}\partial_\theta \Lambda_2(\theta_0)\bigg|^2\bigg] \nonumber \\
&&\quad \leq \frac{C}{nh_n}E\bigg[\bigg|\sum_{l=1}^{L_n-1}(t_0^{l+1}-t_0^l)\acute{\psi}_l(\theta_0)^\intercal
(\hat{Z}_{l-1}^{-1}-Z_{t^{l-1}_0}^{-1})\partial_\theta\tilde{\psi}_l(\theta_0)1_{\{t_0^l<T_{n,\delta}\}}\bigg|^2\bigg] \nonumber \\
&&\quad \quad + \frac{C}{nh_n}\sum_{l=1}^{L_n-1}E\bigg[\bigg(\bigg(\int^{t^{l+1}_0}_{t^l_0}V_tdW_t\bigg)^\intercal(\hat{Z}_{l-1}^{-1}-Z_{t^{l-1}_0}^{-1})\partial_\theta \tilde{\psi}_l(\theta_0)
\bigg)^21_{\{t_0^{l-1}<T_{n,\delta}\}}\bigg] \nonumber \\
&&\quad \quad +O\bigg(\bigg(\frac{L_n}{\sqrt{nh_n}}(c_nh_n)^{3/2}\bigg)^2\bigg) \nonumber \\
&&\quad \leq\frac{C}{nh_n}{\colora \sum_{l=1}^{L_n-1}E\bigg[\bigg|\sum_{m=1}^{c_n}
\acute{\psi}_l(\theta_0)^\intercal Z_{t_0^{l-1}}^{-1}\mfa_{l-1,m}
Z_{t_0^{l-1}}^{-1}\partial_\theta\acute{\psi}_l(\theta_0)
\bigg|^21_{\{t_0^{l-1}<T_{n,\delta}\}}\bigg]} \nonumber \\
&&\quad \quad {\colora + \frac{C}{nh_n}E\bigg[\bigg(c_nh_n\sum_{l=1}^{L_n-1}|\acute{\psi}_l(\theta_0)|\cdot\underbar{R}_{n,pr'/4}((nh_n)^{-1/2})\cdot|\partial_\theta\tilde{\psi}_l(\theta_0)|\bigg)^2\bigg]} \nonumber \\
&&\quad \quad +O\bigg(\frac{L_nc_nh_n}{nh_n}\bigg)\cdot o(c_n^{-1}+(nh_n)^{-1})+O(nh_n^2c_n) \nonumber \\
&&\quad =O\bigg(\frac{1}{nh_n}n\cdot h_n\cdot c_nh_n\bigg)+{\colora +O\bigg(\frac{(L_nc_nh_n)^2}{nh_n}\bigg)\cdot o((nh_n)^{-1})}+o(1)\to 0. \end{eqnarray}
\begin{discuss}
{\colorr We use $(3r+r)\times 2<p$ in the first inequality, $4r+2<p$ in the second inequality, 
and $(4r+1)\times 2 <p$ in the last inequality.}
\end{discuss}

Similarly Lemma~\ref{F-sum-lemma} and (\ref{psi-diff-eq2}) yield
\begin{eqnarray}\label{Lambda4-est2}
&&E\bigg[\bigg|\frac{1}{\sqrt{nh_n}}\partial_\theta \Lambda_4(\theta_0)\bigg|^2\bigg] \nonumber \\
&&\quad \leq \frac{C}{nh_n}E\bigg[\bigg|\sum_{l=1}^{L_n-1}\bigg(\int_{t_0^l}^{t_0^{l+1}}V_tdW_t\bigg)^\intercal \hat{Z}_{l-1}^{-1}
(\partial_\theta\hat{\psi}_l(\theta_0)-\partial_\theta\tilde{\psi}_l(\theta_0))1_{\{t_0^l<T_{n,\delta}\}}\bigg|^2\bigg]+o(1) \nonumber \\
&&\quad \leq O\bigg(\frac{1}{nh_n}L_nc_nh_n(c_n^{-1/2}+\sqrt{c_nh_n})^2\bigg)+o(1)\to 0.
\end{eqnarray}
(\ref{Lambda1-est})--(\ref{Lambda4-est2}) complete the proof.
\end{proof}

Let
  $$\tilde{H}_n(\theta)=-\frac{1}{2}\sum_{l=1}^{L_n-1}{\colord (\Delta_l Y-c_nh_n\tilde{\psi}_l(\theta))^\intercal \frac{Z_{t^{l-1}_0}^{-1}}{c_nh_n}(\Delta_l Y-c_nh_n\tilde{\psi}_l(\theta))}.$$

{\colord The following lemma gives the asymptotic behavior of $\PT^l\tilde{H}_n(\theta)$, which consequently give the asymptotic behavior of $\PT^lH_n(\theta)$.}

\begin{lemma}\label{tildeH-conv-lemma}
{\colorp Let $p,r\geq 1$ such that (\ref{pr-cond}) is satisfied.} Assume (A1-$p$), (A2-$r$), and (A3-$p$) {\colord and that $n^3h_n^5\to 0$ as $n\to\infty$}. Then for any positive numbers $\epsilon_n$ tends to zero,
\begin{equation}\label{tildeH-conv}
\sup_{|\theta-\theta_0|\leq \epsilon_n}\bigg|\frac{1}{nh_n}\partial_\theta^2\tilde{H}_n(\theta)+\Gamma\bigg|\overset{P}\to 0, 
\quad {\rm and} \quad 
\frac{1}{\sqrt{nh_n}}\partial_\theta \tilde{H}_n(\theta_0)\overset{d}\to N(0,\Gamma),
\end{equation}
as $n\to\infty$.
\end{lemma}

\begin{proof}
Let $\PT^l\tilde{\psi}_{l,0}=\PT^l\tilde{\psi}_l(\theta_0)$ for $l\in \{0,1,2,3\}$.
{\colorp Since (A2-$r$) implies
\EQQ{\sup_{|\theta-\theta_0|\leq \epsilon_n}|\partial_\theta \tilde{\psi}_l(\theta)-\partial_\theta \tilde{\psi}_l(\theta_0)|
\leq \epsilon_n \sup_\theta {\colora |\partial_\theta^2 \tilde{\psi}_l(\theta)|\leq C\epsilon_n(1+|X_{t_0^{l-1}}|+|Y_{t_0^{l-1}}|+|Z_{t_0^{l-1}}|)^r,} }
(\ref{delY-eq2}), {\colora (A3-$p$)} and Sobolev's inequality yield}
\EQNN{&\frac{1}{nh_n}\partial_\theta^2\tilde{H}_n(\theta) \\
&\quad =-\frac{1}{nh_n}\sum_{l=1}^{L_n-1}\bigg\{c_nh_n\partial_\theta\tilde{\psi}_l^\intercal(\theta) Z_{t_0^{l-1}}^{-1}\partial_\theta\tilde{\psi}_l(\theta)
-\partial_\theta^2\tilde{\psi}_l^\intercal(\theta) Z_{t_0^{l-1}}^{-1}(\Delta_l Y - c_nh_n\tilde{\psi}_l(\theta))\bigg\} \\
&\quad =-\frac{c_n}{n}\sum_{l=1}^{L_n-1}\PT\tilde{\psi}_l^\intercal(\theta) Z_{t_0^{l-1}}^{-1}\PT\tilde{\psi}_l(\theta)
+\frac{c_n}{n}\SUML \PT^2\tilde{\psi}_l^\intercal(\theta) Z_{t_0^{l-1}}^{-1}(\tilde{\psi}_{l,0}-\tilde{\psi}_l(\theta)) \\
&\quad\quad  {\colorp +\frac{1}{nh_n}\sum_{l=1}^{L_n-1}\partial_\theta^2\tilde{\psi}_l^\intercal(\theta)Z_{t_0^{l-1}}^{-1}\int_{t_0^l}^{t_0^{l-1}}V_sdW_s}+O_P\bigg(\frac{L_n(c_nh_n)^{3/2}}{nh_n}\bigg) \\
&\quad{\colorp = -\frac{c_n}{n}\sum_{l=1}^{L_n-1}\PT\tilde{\psi}_l^\intercal(\theta_0) Z_{t_0^{l-1}}^{-1}\PT\tilde{\psi}_l(\theta_0)
+O_P\bigg(\frac{L_nc_nh_n}{nh_n}\epsilon_n\bigg) 
+O_P\bigg(\frac{\sqrt{L_n}\sqrt{c_nh_n}}{nh_n}\bigg)+o_P(1)} \\
&\quad \overset{P}\to {\colorp -\Gamma},
}
and
\begin{eqnarray}
\partial_\theta \tilde{H}_n(\theta_0)&=& \sum_{l=1}^{L_n-1}\partial_\theta\tilde{\psi}_{l,0}^\intercal Z_{t_0^{l-1}}^{-1}(\Delta_l Y - c_nh_n\tilde{\psi}_{l,0}) \nonumber \\
&=& \sum_{l=1}^{L_n-1}\partial_\theta\tilde{\psi}_{l,0}^\intercal Z_{t_0^{l-1}}^{-1}\int^{t_0^{l+1}}_{t_0^l}V_sdW_s  +O_P(nh_n\sqrt{c_nh_n}). \nonumber 
\end{eqnarray}
We have $O_P(nh_n\sqrt{c_nh_n})={\colorp O_P(\sqrt{nh_n}\cdot\sqrt{nh_n^2c_n})}=o_P(\sqrt{nh_n})$,
\begin{eqnarray}
&&\sum_{l=1}^{L_n-1}E\bigg[\bigg(\frac{1}{\sqrt{nh_n}}\partial_\theta \tilde{\psi}_{l,0}^\intercal Z_{t_0^{l-1}}^{-1}\int_{t_0^l}^{t_0^{l+1}}V_sdW_s\bigg)^2\bigg|\mathcal{F}_{t_0^l}\bigg] \nonumber \\
&&\quad =\frac{1}{nh_n}\sum_{l=1}^{L_n-1}\partial_\theta \tilde{\psi}_{l,0}^\intercal Z_{t_0^{l-1}}^{-1}\int_{t_0^l}^{t_0^{l+1}}E[Z_s|\F_{t_0^l}]dsZ_{t_0^{l-1}}^{-1}\partial_\theta \tilde{\psi}_{l,0} \overset{P}\to \Gamma, \nonumber
\end{eqnarray}
{\colord and}
\begin{eqnarray}
\sum_{l=1}^{L_n-1}E\bigg[\bigg(\frac{1}{\sqrt{nh_n}}\partial_\theta \tilde{\psi}_{l,0}^\intercal Z_{t_0^{l-1}}^{-1}\int_{t_0^l}^{t_0^{l+1}}V_sdW_s\bigg)^4\bigg|\mathcal{F}_{t_0^l}\bigg]
=O_P\bigg(\frac{L_n(c_nh_n)^2}{n^2h_n^2}\bigg)=O_P(L_n^{-1})\overset{P}\to 0. \nonumber
\end{eqnarray}
Then, Lemma 9 in Genon-Catalot and Jacod~(1993) and
the martingale central limit theorem (Corollary 3.1 and the remark after that in Hall and Heyde~(1980)) imply
\begin{equation*}
\frac{1}{\sqrt{nh_n}}\partial_\theta \tilde{H}_n(\theta_0)\overset{d}\to N(0,\Gamma).
\end{equation*}
\end{proof}

\noindent
{\bf Proof of Theorem~\ref{asymp-normal-thm}.}


Similarly to (\ref{H-checkH-diff}), we obtain
\begin{equation}\label{tildeH-diff}
\lim_{\delta\to 0}\liminf_{n\to \infty}P(\tilde{H}_{n,\delta}(\theta)=\tilde{H}_n(\theta) \ {\rm for \ any} \ \theta)=1.
\end{equation}
Since $\partial_\theta H_n(\hat{\theta}_n)=0$ by definition, 
Taylor's formula yields
\begin{equation*}
-\PT H_n(\theta_0)=\PT H_n(\hat{\theta}_n)-\PT H_n(\theta_0)
=\int^1_0\partial_\theta^2H_n(\theta_u)du(\hat{\theta}_n-\theta_0),
\end{equation*}
if $(\theta_u)_{u\in [0,1]}\subset \Theta$, where $\theta_u=u\hat{\theta}_u-(1-u)\theta_0$ for $0\leq u\leq 1$.

Similarly to Lemma~\ref{checkH-tildeH-diff-lemma}, we have
\EQQ{(nh_n)^{-1}\sup_\theta |\PT \Lambda_j(\theta)|\overset{P}\to 0,}
for $1\leq j\leq 4$. 
Then discussions in Section~\ref{consistency-proof-section}, Lemma~\ref{tildeH-conv-lemma}, and (\ref{tildeH-diff}) yield
\begin{equation*}
\frac{1}{nh_n}\int^1_0\partial_\theta^2\check{H}_{n,\delta}(\theta_u)du
=\frac{1}{nh_n}\int^1_0\bigg\{\partial_\theta^2\tilde{H}_{n,\delta}(\theta_u)+\sum_{j=1}^4\partial_\theta^2\Lambda_j(\theta_u)\bigg\}du
\overset{P}\to \Gamma,
\end{equation*}
on $\{\liminf_{n\to\infty}T_{n,\delta}=\infty\}$ for any $\delta>0$, and together with (\ref{H-checkH-diff}) we obtain
\begin{equation}\label{del2H-conv}
\frac{1}{nh_n}\int^1_0\partial_\theta^2H_n(\theta_u)du\overset{P}\to \Gamma.
\end{equation}

Furthermore, Proposition~\ref{score-diff-prop}, Lemma~\ref{tildeH-conv-lemma} and (\ref{tildeH-diff}) yield
\begin{eqnarray}
\frac{1}{\sqrt{nh_n}}\partial_\theta \check{H}_{n,\delta}(\theta_0)
=\frac{1}{\sqrt{nh_n}}\partial_\theta \tilde{H}_{n,\delta}(\theta_0)+o_P(1)\overset{d}\to N(0,\Gamma),
\end{eqnarray}
on $\{\liminf_{n\to\infty}T_{n,\delta}=\infty\}$. Then together with (\ref{del2H-conv}) and (\ref{H-checkH-diff}), we obtain
\begin{eqnarray}
\sqrt{nh_n}(\hat{\theta}_n-\theta_0)=\Gamma^{-1}\frac{1}{\sqrt{nh_n}}\partial_\theta H_n(\theta_0)+o_P(1)\overset{d}\to N(0,\Gamma^{-1}). \nonumber
\end{eqnarray}
\qed

\noindent 
{\bf Acknowledgements}
Teppei Ogihara was supported by Japan Society for the Promotion of Science KAKENHI Grant Numbers 19K14604 and 21H00997, Japan.

\bibliographystyle{abbrv}

\end{document}